\documentclass{birkjour}
\usepackage{amsmath, amsthm, amscd, amsfonts,url, amssymb,graphicx,graphics, color}
\usepackage[bookmarksnumbered, plainpages]{hyperref}
 \newtheorem{thm}{Theorem}[section]
 \newtheorem{cor}[thm]{Corollary}
 \newtheorem{lem}[thm]{Lemma}
 \newtheorem{prop}[thm]{Proposition}
 \theoremstyle{definition}
 \newtheorem{defn}[thm]{Definition}
 \newtheorem{rem}[thm]{Remark}
 
 \numberwithin{equation}{section}
\usepackage{pgf,tikz}
\usetikzlibrary{arrows}

\begin{document}

\title[Cayley properties of  the line graphs induced by  ...]
 {Cayley properties of  the line graphs induced\\ by consecutive layers of  the hypercube}

\author[S.Morteza Mirafzal]{S.Morteza Mirafzal}

\address{Department of Mathematics\\ Lorestan University\\ Khorramabad\\ Iran}

\email{smortezamirafzal@yahoo.com}
\email{mirafzal.m@lu.ac.ir}

\thanks{}
\subjclass{Primary 05C25}

\keywords{Middle layer cube, Line graph,  Automorphism
group, $k$-homogeneous permutation group,   Sharply $2$-transitive permutation group,  Cayley graph}

\date{}
\begin{abstract}
Let $n >3$ and $  0< k < \frac{n}{2} $  be integers. In this paper, we  investigate some algebraic properties of  the  line graph of the graph $ {Q_n}(k,k+1) $  where $ {Q_n}(k,k+1) $  is the subgraph of the hypercube $Q_n$  which is induced by the set of vertices of weights $k$ and $k+1$. In the first step, we determine the automorphism groups of these graphs for all values of $n,k$. In the second step,
   we   study  Cayley properties of the line graphs of these graphs. In particular,  we show that  if $k\geq 3$  and $  n \neq 2k+1$, then  except for the cases $k=3, n=9$ and $k=3, n=33$, the line graph of the graph $ {Q_n}(k,k+1)  $   is a vertex-transitive non-Cayley graph. Also, we show that the line graph of the graph   $ {Q_n}(1,2) $ is a Cayley graph if and only if $ n$ is  a power of a prime $p$. Moreover, we show that for \lq{}almost all\rq{} even values of    $k$, the line graph of the graph   $ {Q_{2k+1}}(k,k+1) $ is a vertex-transitive non-Cayley graph.
\end{abstract}

\maketitle ----------------------------------------------------------------
\section{ Introduction}
\noindent
 In this paper, a graph $\Gamma=(V,E)$ is
considered as an undirected simple graph where $V=V(\Gamma)$ is the vertex-set
and $E=E(\Gamma)$ is the edge-set. For all the terminology and notation
not defined here, we follow [2,8,9].\

Let $ n \geq 1 $ be an integer. The hypercube  $Q_n$ is the graph whose vertex set is $ \{0,1  \}^n $, where two $n$-tuples  are adjacent if  they differ in precisely one coordinates.
The   hypercube $Q_n$,  has been extensively studied. Nevertheless,
many open questions remain. Harary, Hayes, and Wu [11] wrote a comprehensive survey on hypercube graphs.  In the graph $Q_n$, the layer $L_k$ is the set of vertices which contain $k$ 1's, namely,  vertices of weight $k$, $ 1 \leq k \leq n$.  We denote by $ {Q_n}(k,k+1)$, the subgraph of $Q_n$ induced by layers $L_k$ and $ L_{k+1} $. For $ n=2k+1 $,   the graph $ {Q_{2k+1}}(k,k+1) $ has been investigated from various aspects,  by various authors and is called the middle layer cube [7,11,13,30,31] or regular hyperstar graph (denoted by $HS(2k,k)$) [18,19,23].   It has been conjectured by Dejter, Erd\H{o}s, and Havel [13] among others, that $ Q_{2k+1}(k,k+1) $
is hamiltonian.   Recently, M\"utze [31]  showed that the    middle layer cube $ {Q_{2k+1}}(k,k+1) $ is hamiltonian.\

 Figure 1.  shows the graph  $HS(6,3)$ ($ Q_{5}(2,3) $) in plane. Note that in this figure the set $\{i,j,k\} $ ($\{ i,j \}$) is denoted by $ijk$ ($ij$). \

\definecolor{qqqqff}{rgb}{0.,0.,1.}
\begin{tikzpicture}[line cap=round,line join=round,>=triangle 45,x=.65cm,y=.80cm]
\clip(-4.3,-2.38) rectangle (11.32,6.3);
\draw (-0.9,3.74) node[anchor=north west] {13};
\draw (0.9,5.6) node[anchor=north west] {123};
\draw (4.88,5.3) node[anchor=north west] {124};
\draw (5.7,3.78) node[anchor=north west] {24};
\draw (5.7,1.9) node[anchor=north west] {245};
\draw (5.0,0.66) node[anchor=north west] {45};
\draw (3.0,0.04) node[anchor=north west] {453};
\draw (1.14,0.52) node[anchor=north west] {43};
\draw (-0.8,1.68) node[anchor=north west] {143};
\draw (3.28,5.9) node[anchor=north west] {12};
\draw (4.92,4.76)-- (3.06,5.52);
\draw (3.06,5.52)-- (1.14,4.88);
\draw (1.14,4.88)-- (0.12,3.52);
\draw (0.12,3.52)-- (0.22,1.64);
\draw (0.22,1.64)-- (1.24,0.64);
\draw (1.24,0.64)-- (3.44,0.2);
\draw (3.44,0.2)-- (4.88,0.64);
\draw (5.72,3.38)-- (5.6,3.34);
\draw (4.92,4.76)-- (5.6,3.34);
\draw (5.72,1.68)-- (5.66,1.82);
\draw (5.66,1.82)-- (5.6,3.34);
\draw (5.66,1.82)-- (4.88,0.64);
\draw (3.96,3.8)-- (3.24,4.48);
\draw (3.96,3.8)-- (4.02,1.8);
\draw (4.02,1.8)-- (3.44,1.28);
\draw (1.64,3.54)-- (1.66,1.82);
\draw (4.86,3.38)-- (4.8,1.88);
\draw (0.7,3.82) node[anchor=north west] {23};
\draw (0.6,2.12) node[anchor=north west] {234};
\draw (4.68,3.86) node[anchor=north west] {14};
\draw (4.3,1.8) node[anchor=north west] {145};
\draw (3.32,5.) node[anchor=north west] {125};
\draw (3.5,4.48) node[anchor=north west] {25};
\draw (3.52,2.44) node[anchor=north west] {235};
\draw (2.84,1.18) node[anchor=north west] {35};
\draw (1.96,1.4) node[anchor=north west] {135};
\draw (2.0,4.2) node[anchor=north west] {15};
\draw (4.86,3.38)-- (0.22,1.64);
\draw (5.6,3.34)-- (1.66,1.82);
\draw (2.62,3.54)-- (2.58,3.54);
\draw (2.58,3.54)-- (3.24,4.48);
\draw (2.7,1.66)-- (2.68,1.46);
\draw (2.68,1.46)-- (2.58,3.54);
\draw (3.44,1.28)-- (2.68,1.46);
\draw (1.64,3.54)-- (4.02,1.8);
\draw (1.66,1.82)-- (1.24,0.64);
\draw (2.68,1.46)-- (0.12,3.52);
\draw (1.64,3.54)-- (1.14,4.88);
\draw (2.58,3.54)-- (4.8,1.88);
\draw (3.24,4.48)-- (3.06,5.52);
\draw (4.92,4.76)-- (4.86,3.38);
\draw (3.96,3.8)-- (5.66,1.82);
\draw (4.8,1.88)-- (4.88,0.64);
\draw (3.44,1.28)-- (3.44,0.2);
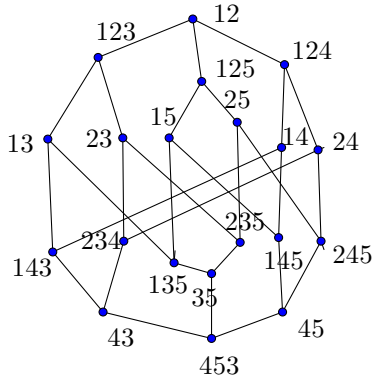
\draw (-2.2,-0.6) node[anchor=north west] {Figure 1. The regular hyperstar graph  HS(6,3)};
\begin{scriptsize}
\draw [fill=qqqqff] (0.12,3.52) circle (1.5pt);
\draw [fill=qqqqff] (0.22,1.64) circle (1.5pt);
\draw [fill=qqqqff] (1.24,0.64) circle (1.5pt);
\draw [fill=qqqqff] (3.44,0.2) circle (1.5pt);
\draw [fill=qqqqff] (4.88,0.64) circle (1.5pt);
\draw [fill=qqqqff] (1.14,4.88) circle (1.5pt);
\draw [fill=qqqqff] (3.06,5.52) circle (1.5pt);
\draw [fill=qqqqff] (4.92,4.76) circle (1.5pt);
\draw [fill=qqqqff] (5.6,3.34) circle (1.5pt);
\draw [fill=qqqqff] (5.66,1.82) circle (1.5pt);
\draw [fill=qqqqff] (1.64,3.54) circle (1.5pt);
\draw [fill=qqqqff] (1.66,1.82) circle (1.5pt);
\draw [fill=qqqqff] (4.86,3.38) circle (1.5pt);
\draw [fill=qqqqff] (4.8,1.88) circle (1.5pt);
\draw [fill=qqqqff] (3.24,4.48) circle (1.5pt);
\draw [fill=qqqqff] (3.44,1.28) circle (1.5pt);
\draw [fill=qqqqff] (3.96,3.8) circle (1.5pt);
\draw [fill=qqqqff] (4.02,1.8) circle (1.5pt);
\draw [fill=qqqqff] (2.58,3.54) circle (1.5pt);
\draw [fill=qqqqff] (2.68,1.46) circle (1.5pt);
\end{scriptsize}
\end{tikzpicture}

The study of vertex-transitive graphs has a long and rich history in discrete
mathematics. Prominent examples of vertex-transitive graphs are Cayley graphs
which are important in both theory as well as applications. Vertex-transitive graphs
that are not Cayley graphs, for which we use the abbreviation VTNCG,
have been an object of a systematic study since   1979 [3,10].
In trying to recognize
whether or not a vertex-transitive graph is a Cayley graph, we are left with
the problem of determining whether the automorphism group contains a regular
subgroup [2]. The reference   [1] is an  excellent  source for studying graphs that are VTNCG.\
In particular, determining the automorphism group of a given graph can be very useful in determining whether this graph is a Cayley graph. In this area of research, in algebraic graph theory,  there are various works and  some of  the recent  papers  in this scope are  [3,10,14,15,16,22,23,24,25,27,28,29]. \
In this paper, we   investigate some algebraic properties of  the  line graph of the graph $ {Q_n}(k,k+1) $, in particular, we   study cayleyness of this graph.

 We can consider the graph $Q_n$ from another point of view. The Boolean lattice $BL_n, n \geq 1$, is the graph whose vertex set is the set of all subsets of $[n]= \{ 1,2,...,n \}$, where two subsets $x$ and $y$ are adjacent if their symmetric difference has precisely one element.  In the graph $BL_n$, the layer $L_k$ is the set of $k$-subsets of $[n]$.  We denote by $ {BL_n}(k,k+1)$, the subgraph of $BL_n$ induced by layers $L_k$ and $ L_{k+1} $.

 Note that  if $A$ is  a subset of $[n]$,  then the
characteristic function of $A$ is the function $ \chi_{A} :[n]
\longrightarrow \{0,1\}$,  with the rule   $\chi_{A}(x) =1$, if and only if
$x\in A$. We now can show that the mapping  $\chi :V(BL_n)  \longrightarrow V(Q_n),  $ defined by the rule, $\chi ( A)= \chi_{A},  $  is a graph isomorphism. Now, it is clear that the graph $Q_n$ is isomorphic with the graph $BL_n$, by an isomorphism that induces an isomorphism from ${BL_n}(k,k+1)$  to  $ {Q_n}(k,k+1)$.   For this reason, in the sequel,  we work on the graph  ${BL_n}(k,k+1)$ and for abbreviation, we denote it by $ B(n,k)  $.  We know that $ n \choose k$ =$ n \choose n-k$, so $ B(n,k)   \cong   B(n,n-k) $.   Therefore, in the sequel we assume that $ k < \frac {n}{2}  $.

\section{Preliminaries}

 The group of all permutations of a set $V$ is denoted by S$ym(V)$  or
just S$ym(n)$ when $ | V | =n $. A $permutation\  group$ $G$ on
$V$ is a subgroup of S$ym(V)$. In this case we say that $G$ act
on $V$. If $\Gamma$ is a graph with vertex-set $V$, then we can view
each automorphism of $\Gamma$ as a permutation of $V$, and so $Aut(\Gamma)$ is a
permutation group. Let the group $G$ act  on $V$, we say that $G$ is
$transitive$ (or $G$ acts $transitively$  on $V$)   if there is just
one orbit. This means that given any two element $u$ and $v$ of
$V$, there is an element $ \beta $ of  $G$ such that  $\beta (u)= v
$.

The graph $\Gamma$ is called $vertex$-$transitive,$  if  $Aut(\Gamma)$
acts transitively on $V(\Gamma)$. The action of $Aut(\Gamma)$ on
$V(\Gamma)$ induces an action on $E(\Gamma)$, by the rule
$\beta\{x,y\}=\{\beta(x),\beta(y)\}$,  $\beta\in Aut(\Gamma)$, and
$\Gamma$ is called $edge$-$transitive$ if this action is
transitive. The graph $\Gamma$ is called $symmetric,$  if  for all
vertices $u, v, x, y,$ of $\Gamma$ such that $u$ and $v$ are
adjacent, and $x$ and $y$ are adjacent, there is an automorphism
$\alpha$ such that $\alpha(u)=x$,   and $ \alpha(v)=y$. It is clear
that a symmetric graph is vertex-transitive and edge-transitive.

For $v\in V(\Gamma)$ and $G=Aut(\Gamma)$, the stabilizer subgroup
$G_v$ is the subgroup of $G$ containing all automorphisms which
fix $v$. In the vertex-transitive case all stabilizer subgroups
$G_v $ are conjugate in $G$, and consequently are  isomorphic. In this
case, the index of $G_v$ in $G$ is given by the equation,  $| G
: G_v | =\frac{| G |}{| G_v |} =| V(\Gamma)|
$. If each stabilizer $ G_v $ is the identity group, then every
element of $G$, except the identity, does not fix any vertex, and
we say that $G$ acts $semiregularly$ on $V$. We say that $G$ act
$regularly$ on $V$ if and only if $G$ acts  transitively and
semiregularly on $ V$,
 and in this case we
have $| V | = | G |$.\\
Let  $n,k \in \mathbb{ N}$ with $ k < n,   $ and let $[n]=\{1,...,n\}$. The $Johnson\  graph$ $J(n,k)$ is defined as the graph whose vertex set is $V=\{v\mid v\subseteq [n], |v|=k\}$ and two vertices $v$,$w, $ are adjacent if and only if $|v\cap w|=k-1$.   The Johnson\ graph $J(n,k)$ is a  vertex-transitive graph  [9].   It is an easy task to show that the set  $H= \{ f_\theta \ | \  \theta \in$ S$ym([n]) \} $,  $f_\theta (\{x_1, ..., x_k \}) = \{ \theta (x_1), ..., \theta (x_k) \} $,    is a subgroup of $ Aut( J(n,k) ) $ [9].   It has been shown that   $Aut(J(n,k)) \cong$ S$ym([n])$, if  $ n\neq 2k, $  and $Aut(J(n,k)) \cong$ S$ym([n]) \times \mathbb{Z}_2$, if $ n=2k$,   where $\mathbb{Z}_2$ is the cyclic group of order 2 [5,28].

Let $G$ be any abstract finite group with identity $1$, and
suppose $\Omega$ is a set of   $G$, with the
properties:

(i) $x\in \Omega \Longrightarrow x^{-1} \in \Omega$,   $ \ (ii)
 \ 1\notin \Omega $.

The $Cayley\  graph$  $\Gamma=\Gamma (G; \Omega )$ is the (simple)
graph whose vertex-set and edge-set defined as follows:

$V(\Gamma) = G $,  $  E(\Gamma)=\{\{g,h\}\mid g^{-1}h\in \Omega \}$.
It can be shown that a  connected graph $\Gamma$ is a Cayley graph
if and only if $Aut(\Gamma)$ contains a subgroup $H$,  such that
$H$ acts regularly on $V(\Gamma)$ [2, chapter 16].
\section{Main results}
\begin{defn}
Let $ n \geq 4 $ be an integer and $ [n] = \{1,2,..., n \} $. Let $ k$ be an integer such that
$1\leq k <\frac{n}{2}$. The graph $ B(n,k)$ is a
 graph with the vertex set  $V=\{v \  |  \ v \subset [n] ,  | v |  \in \{ k,k+1  \} \} $ and the
edge set $ E= \{ \{ v , w \} \  | \  v , w \in V , v \subset w $ or $ w \subset v \} $.
\end{defn}
According to the Definition 3.1.  Figure 2.  shows  $ B(5,1)$    in   the plane.\

\definecolor{qqqqff}{rgb}{0.,0.,1.}
\begin{tikzpicture}[line cap=round,line join=round,>=triangle 45,x=-.55cm,y=.8cm]
\clip(-4.46,-0.04) rectangle (10.9,4.78);
\draw (-2.,4.)-- (-4.32,1.98);
\draw (-2.,4.)-- (-2.7,1.98);
\draw (-2.,4.)-- (-1.22,2.);
\draw (-2.,4.)-- (0.24,2.);
\draw (0.,4.)-- (-4.32,1.98);
\draw (0.,4.)-- (1.56,2.);
\draw (0.,4.)-- (3.18,2.02);
\draw (0.,4.)-- (4.56,1.96);
\draw (2.,4.)-- (-2.7,1.98);
\draw (2.,4.)-- (1.56,2.);
\draw (2.,4.)-- (6.,2.);
\draw (2.,4.)-- (7.38,1.98);
\draw (4.,4.)-- (-1.22,2.);
\draw (4.,4.)-- (3.18,2.02);
\draw (4.,4.)-- (6.,2.);
\draw (4.,4.)-- (8.54,1.98);
\draw (6.,4.)-- (0.24,2.);
\draw (6.,4.)-- (4.56,1.96);
\draw (6.,4.)-- (7.38,1.98);
\draw (6.,4.)-- (8.54,1.98);
\draw (4.0,0.95) node[anchor=north west] {Figure 2.  B(5,1)};
\begin{scriptsize}
\draw [fill=qqqqff] (-2.,4.) circle (1.5pt);
\draw[color=qqqqff] (-1.84,4.48) node {$ 1 $};
\draw [fill=qqqqff] (0.,4.) circle (1.5pt);
\draw[color=qqqqff] (0.18,4.4) node {$2$};
\draw [fill=qqqqff] (2.,4.) circle (1.5pt);
\draw[color=qqqqff] (2.14,4.4) node {$3$};
\draw [fill=qqqqff] (4.,4.) circle (1.5pt);
\draw[color=qqqqff] (4.18,4.42) node {$4$};
\draw [fill=qqqqff] (6.,4.) circle (1.5pt);
\draw[color=qqqqff] (6.2,4.42) node {$5$};
\draw [fill=qqqqff] (-4.32,1.98) circle (1.5pt);
\draw[color=qqqqff] (-4.34,1.56) node {$12$};
\draw [fill=qqqqff] (-2.7,1.98) circle (1.5pt);
\draw[color=qqqqff] (-2.7,1.64) node {$13$};
\draw [fill=qqqqff] (-1.22,2.) circle (1.5pt);
\draw[color=qqqqff] (-1.28,1.66) node {$14$};
\draw [fill=qqqqff] (0.24,2.) circle (1.5pt);
\draw[color=qqqqff] (0.22,1.68) node {$15$};
\draw [fill=qqqqff] (1.56,2.) circle (1.5pt);
\draw[color=qqqqff] (1.56,1.66) node {$23$};
\draw [fill=qqqqff] (3.18,2.02) circle (1.5pt);
\draw[color=qqqqff] (3.18,1.72) node {$24$};
\draw [fill=qqqqff] (4.56,1.96) circle (1.5pt);
\draw[color=qqqqff] (4.58,1.66) node {$25$};
\draw [fill=qqqqff] (6.,2.) circle (1.5pt);
\draw[color=qqqqff] (5.96,1.66) node {$34$};
\draw [fill=qqqqff] (7.38,1.98) circle (1.5pt);
\draw[color=qqqqff] (7.32,1.64) node {$35$};
\draw [fill=qqqqff] (8.54,1.98) circle (1.5pt);
\draw[color=qqqqff] (8.54,1.66) node {$45$};
\end{scriptsize}
\end{tikzpicture} \
  \newline
Note that in the above figure $i= \{  i\}, ij= \{ i,j \}$. \

\

\begin{rem} In the sequel,  we denote every set $ \{ x_1,x_2,  ... ,x_t   \} $ by $ x_1 x_2 ... x_t   $. \
\end{rem}
\

We see that in  $ \Gamma = B(n,k)$, if $ v = x_1 ... x_k \in P_1 = \{ v \ | \  v\subset [n], |
 v |= k \} $, then
 $$ N ( v ) = \{ x_1. .. x_ky_1, ..., x_1 ... x_ky_{n-k} \}, $$
 where $  \{ x_1,. .. ,x_k, y_1, ..., y_{n-k}\} = [n] =\{1,...,n \} $.
 Hence,  $ deg (v) = | N (v) | = n-k $.
 On the other hand,  if $ w = x_1 ...x_k x_{k+1} \in P_2 = \{ v \  | \
 v  \subset  [n], | v | = k + 1 \} $,  then 
$$ N ( w ) = \{ u \mid u \subset w,  | u | =k \}, $$
 and hence $| N ( w )| = deg ( w ) = k+1 $.  Therefore, if $ k \neq \frac { n-1 } {2 }$,
then we have $ k+1 \neq n-k $, and thus if $ k \neq \frac { n-1 } {2 }, $  then the graph $ B(n,k)$ is not a regular graph.\

 Since every vertex of $ B(n,k)$ which is in $ P_1 $ is of
 degree $ n-k $ and $ | P_1 | = { n \choose k } $ then
 the number of edges of $ B(n,k)$ is $ ( n - k ){ n \choose k } $ and
 the number of vertices of $ B(n,k)$ is  $ { n \choose k } + { n \choose k+1 } $.

\begin{prop}
The graph $ B(n,k) $ is bipartite and connected.

\end{prop}

\begin{proof} Let $ P_1 = \{  v\ |\  v  \in V(B(n,k)), |v| =k\}  $ and $ P_2 = \{  v \ |  \ v \in  V(B(n,k)), |v|$ =$ k+1\}   $. Then from the definition of the graph $B(n,k)$
 it follows that  $ V = V(B(n,k)) = P_1 \cup P_2 $,
 $ P_1 \cap P_2  = \emptyset $ and every edge $ e=\{ x , y \} $ of $ B(n,k) $ is
 such that only one of  $ x $ or $ y $ is in $ P_1 $ and the other is in $ P_2 $. \

 We now show that $ B(n,k) $ is a connected graph. Let $ x , y $ be two vertices of $B(n,k) $.
 In the first step,  let  $ x , y $  be  in $ P_1 $. Let $ x=x_1 x_2 ...x_k $, $y=y_1y_2 ...y_k $
and $ | x_1 ...x_k \cap y_1...y_k| = k-t $, $ 0 \leq t \leq k-1 $. We can show by induction on $ t $
that $ d ( x , y) \leq 2t $, where $ d ( x , y) $ is the distance of vertices $ x $  and $ y $
 in $ B(n,k) $. Let $ | x \cap y |= | x_1 ...x_k \cap y_1 ...y_k |=k-1 $, then
 we have $ x = x_1 ...x_{k-1} u $, $ y= x_1 ...x_{k-1} v $,
 for some $ u, v \in [n] = \{ 1,...n \} $, $ v \neq u $. Now if $ z = x_1 ...x_{k-1} u v $ then $ P : x, z, y $ is
 a path between $ x $ and $ y $ and we have $ d(x, y )= 2 = 2t $,  for $ t = 1 $. \

Now, suppose that the assertion is true for $t= m $,  where $ 1 \leq m < k-1 $. \
 Let $ |x\cap y | =k-( m+1)$. Let $ x = x_1 ... x_{k-m-1} u_1 ... u_{m+1} $, $ y = x_1 ... x_{k-m-1} v_1 ... v_{m+1} $.
Then  for the vertex,  $ z_1= x_1 ... x_{k-m-1}$$ u_1 ... u_m v_1 $,   we have $ |z_1 \cap y | =k-m,$  $| z_1 \cap x| =k-1 $,  hence
 by the assumption of induction we have $ d(x, z_1 )=2,  $ and $ d(z_1, y)=2m $, therefore
 $ d(x, y)  \leq  d(x, z_1)  +  d(z_1, y) = 2 + 2m = 2(m+1)  $(in fact,  we can show that in this case,  $ d(x, y)= 2(m+1) $). \

In the second step,  let $ x \in P_1 $ and $ y \in P_2 $.
 If $ x= x_1 x_2 ... x_k $, $ y= y_1 y_2 ... y_k y_{k+1} $,
 then $ z= y_1 ... y_k \in P_1 $ and $ z $ is adjacent to $ y. $  Now,  by
 what we have seen in the first step,  there is a path between $ x $ and $ z $ in $ B(n,k)$, and
 therefore there is path between $ x $ and $ y $  in $ B(n,k) $. \

  In the last step,  let $ x, y \in P_2 $. If $ x = x_1...x_k x_{k+1} $, $ y = y_1...y_k x_{k+1} $.
Then,  for $ z = x_1 ...x_k $ we have $ z \in P_1 $ and $ z $ is adjacent to $ x $, and  so according to the
 second step,  there is a path between $ y $ and $ z $, and  therefore there
 is a path between $ x $ and $ y $.
\end{proof}

By the method which we used in the proof of Proposition 3.3.  we can deduce the following result.
\

\

\begin{cor}
Let $ D $ be the diameter of the graph $ B(n,k). $  If $n \neq 2k+1,$ then $ D=2(k+1)=2k+2 $. If $n=2k+1, $ then $D=2k+1$.
\end {cor}

We know that  every vertex-transitive graph is a  regular graph,     so if $ \Gamma $ is not a regular graph,  then $ \Gamma $ is
not a vertex-transitive graph. Thus, if $ n\neq 2k+1$, then $B(n,k)$ is not a vertex-transitive graph.  \

Let $ V= V(B(n,k)) $ be the vertex set of $ B(n,k) $.
  Then,  for each
$\sigma \in$  $Sym([n]) $, the mapping,
 $$ f_{ \sigma } :V\longrightarrow V,  \
  f_{ \sigma }(v) =\{\sigma ( x )| \  x \in v \}, \  v\in V, $$
 is a bijection of $ V $  and  $f_{ \sigma }$ is an automorphism
 of the graph $ B(n,k)  $. In fact,  for each edge $ e =\{  v,w \}$=$\{x_1 ...x_k, x_1...x_k x_{k+1} \} $, we have, $$ f_{ \sigma } ( e ) =\{f_{ \sigma }(v),f_{ \sigma }(w)   \} =   \{{\sigma} ( x_1 )...{\sigma} ( x_k ), {\sigma} ( x_1 )...{\sigma} ( x_k ) {\sigma} ( x_{k+1} ) \},  $$
 and consequently $f_{ \sigma } ( e ) $ is an edge of $ B(n,k)  $. Similarly, if $ f = \{ x, y \} $ is
 not an edge of  $B(n,k),  $
 then $ f_{ \sigma } ( f ) = \{ f_{ \sigma }(x),f_{ \sigma }(y) \} $ is not an
 edge of $ B(n,k) $. Therefore,  if $ H =\{ f_{ \sigma } \mid { \sigma } \in$ S$ym ( [n] ) \} $,
 then $ H $ is a subgroup of the group $ G = Aut (B(n,k) ) $. In fact, we show that if $n \neq 2k+1$, then $ Aut ( B(n,k)) = H, $
  and if $n=2k+1$, then $ Aut ( B(n,k)) = H \times \mathbb{Z}_2 $, where $\mathbb{Z}_2$ is the cyclic group of order 2.
 It is  clear   that $ H \cong$ S$ym ([n]) $.

 \begin{prop} If $ \Gamma =B(n,k) $,
  then $ \Gamma $ is edge-transitive. Moreover,  if $n=2k+1$, then $ \Gamma $ is vertex-transitive.

\end{prop}

\begin{proof}\ If $ e_1 = \{ x_1 ...x_k, x_1 ...x_kx_{k+1} \}, \  e_2 = \{y_1 ...y_k, y_1 ... y_k y_{k+1} \} $
are edges of $ \Gamma $,  then we define  the mapping   \\
 $$ \theta = { {x_1, ... ,x_k ,  x_{k+1}, u_1,  ...,  u_{n-k-1}} \choose  { y_1, ..., y_k,  y_{k+1}, v_1, ..., v_{n-k-1}}},  $$ \\
 where $ \{ x_1, ... x_{ k+1 }, u_1, ..., u_{n-k-1} \} = \{ 1, ...,n \} = \{ y_1, ..., y_{k+1}, v_1,..., v_{n-k-1} \} $.\newline
It is an easy task to  show that $ \theta   \in$ S$ym ( [n] ) $. Therefore,
 $ f_{\theta} \in S= \{ f_{\sigma } \mid \sigma \in$ S$ym ( [n] )\} \leq Aut ( \Gamma ) $,
 and we have $ f_{\theta}(e_1) = e_2 $. \

We now assume that $n=2k+1$. For each  vertex $v$ in $V=V(B(n,k))$,  let $v^c$ be the complement of the
  set $v$ in $[n]$. We define the mapping  $\alpha :V \longrightarrow V $ by the
  rule,  $ \alpha(v) =v^c $, for every $v$ in $V$. Since the complement of a $k$-subset of the set $[n]$ is a $(k+1)$-subset of $[n]$, then $\alpha$ is a well-defined mapping. We can  see, by an easy argument that $ \alpha
  $ is an automorphism of $ B(n,k) $, namely,  $ \alpha \in Aut(B(n,k)) $, also    $\alpha$ is of order 2. We  recall that the  the group S$ym([n])$ acts   transitively  on the set of all $k$-subsets of $[n]$.
Let $v,w$ are vertices in $B(n,k)$. If $v,w$ are $k$-subsets of $[n]$, then there is some $\theta \in$ S$ym([n])$ such that $f_{\theta}(v) =w$. If $v$ is a $k$-subset and $w$ is a $k+1$-subset of $[n]$, then $\alpha(w) $ is a $k$-subset of $[n]$, and hence there is some $ f_{\theta} \in Aut(B(n,k))$ such that $f_{\theta}(v)=\alpha(w)$, and thus $ (\alpha f_{\theta})(v)=w $.
\end{proof}
\

In the sequel,  we  determine
 $ Aut (B(n,k)) $, the automorphism group of the graph $ B(n,k) $.

\begin{lem} Let $n$ and $k$ be integers with  $\frac{n}{2}> k\geq1$, and
let $\Gamma= (V,E)= B(n,k)$,  with $V=P_1 \cup P_2 $, $P_1 \cap P_2 =\emptyset$, where $ P_1= \{ v \  |  \  v\subset [n], |v|=k \}$ and  $ P_2= \{ w \ |\  w \subset [n], |w|=k+1 \}$. If $f$ is an automorphism of $ \Gamma$ such that $f(v)=v$ for every $v\in P_1$, then $f$ is the identity automorphism of $ \Gamma$.

\end{lem}

\begin{proof}

The proof is straightforward (see [27] Lemma 3.1).

\end{proof}

\begin{rem} If in the assumptions of the  above lemma,  we replace  with  $f(v)=v$ for every $v\in P_2$, then we can show, by a similar discussion,   that $f$ is the identity automorphism of $ \Gamma$.
\end{rem}

\begin{lem} Let $\Gamma =(V, E)$ be a connected  bipartite graph with partition $V=V_1 \cup V_2$, $V_1 \cap V_2 = \emptyset $. Let $ f$ be an automorphism of $\Gamma $.  If  for a fixed vertex $v \in V_1 $, we have $ f(v) \in V_1$, then  $f(V_1) = V_1$ and $f(V_2) =V_2$.  Also, if  for a fixed vertex $v \in V_1 $,   we have $ f(v) \in V_2$,  then  $f(V_1) = V_2$ and $f(V_2) =V_1$.
\end{lem}

\begin{proof}
The proof is straightforward (see [27] Lemma 3.3).
\end{proof}
We now can conclude the following result.

\begin{cor}
Let $n$ and $k$ be integers with  $\frac{n}{2}> k\geq1$, and
let $\Gamma= (V,E)= B(n,k)$,  with $V=P_1 \cup P_2 $, $P_1 \cap P_2 =\emptyset$, where $ P_1= \{ v \  |  \  v\subset [n], |v|=k \}$ and  $ P_2= \{ w \ |\  w \subset [n], |w|=k+1 \}. $  Let $f$ be an automorphism of the graph $\Gamma$. If $n \neq 2k+1$,   then $ f(V_1)=V_1$ and $f(V_2) =V_2$.     Also, if $n = 2k+1$,  then $ f(V_1)=V_1$ and $f(V_2) =V_2$,  or $ f(V_1) = V_2 $ and $f(V_2) = V_1$.
\end{cor}

\begin{proof}
We know that if $v \in P_1,$  then $deg(v)=n-k,  $  and if $w \in P_2,$  then $deg(w)=k+1$. If $n \neq 2k+1$,   then $ n-k \neq k+1.  $  Hence,  if $v \in P_1,$   then $f(v) \notin P_2$ (note that $deg(f(v))=deg(v)$). Therefore, if $v \in P_1,$ then  $f(v) \in P_1.$  Now, since by proposition 3.3. the graph $B(n,k)$ is connected and bipartite,  thus by by Lemma 3.8. we conclude that  $ f(V_1)=V_1$ and $f(V_2) =V_2$. The proof of the second assertion is similar.
\end{proof}

We now are ready to prove one of the important results of this paper.

\begin{thm}

Let $n$ and $k$ be integers with  $\frac{n}{2}> k\geq1$, and
let $\Gamma= (V,E)= B(n,k), $  with partition $V=P_1 \cup P_2 $, $P_1 \cap P_2 = \emptyset$, where $ P_1= \{ v \  | \  v\subset [n], |v|=k \}$ and  $ P_2= \{ w \  |  \ w \subset [n], |w|=k+1 \}$. If $n \neq 2k+1$, then $Aut(\Gamma) \cong$  $Sym([n]) $, and if $n=2k+1, $ then  $Aut(\Gamma) \cong$       $Sym([n]) \times \mathbb{Z}_2$,  where $\mathbb{Z}_2$ is the cyclic group of order $2$.

\end{thm}

\begin{proof}
The proof is almost similar to what is seen in [27, Theorem 3.6]. \newline
(a)\ \ In the first step, we prove the theorem for the case   $ n\neq 2k+1 $.  Let  $ H =\{ f_\theta \  |\  \theta \in$   $Sym ([n]) \} $.  We have already   seen  that $H  \cong$   $Sym([n]) $ and $H \leq Aut(B(n,k)) $.
Let $G=Aut(B(n,k)).$  We   show that  $G=H.$ Let  $f \in G. $ Then by Corollary 3.9. we have
 $f(V_1)=V_1$. Then, for every vertex $v \in V_1$ we have $f(v) \in V_1$, and therefore the mapping $ g=f_{|V_1}: V_1 \rightarrow  V_1$, is a permutation of $V_1$ where $ f_{|V_1}  $ is the restriction of $f$ to $V_1$. Let $ \Gamma_2= J(n,k)$ be the Johnson graph with the vertex set $V_1$. Then, the vertices $ v,w \in V_1$ are adjacent in $\Gamma_2$  if and only if $| v \cap w| =k-1$.\

We assert that the permutation  $ g= f_{|V_1}  $ is an automorphism of the graph $\Gamma_2$. \newline
For proving our assertion, it is sufficient to show that if $v,w \in V_1$ are such that $| v \cap w| =k-1, $ then we must have  $ |g(v) \cap g(w)|= k-1$. Note that since $v,w$ are $k$-subsets of $[n]$, then if $u$ is a common neighbor
of $v,w$ in the   graph $ \Gamma=B(n,k)$, then the set $u$ contains the sets $v$ and $w$. In particular $u$ contains the $(k+1)$-subset $ v\cup w$. now, since $u$ is a   $(k+1)$-subset  of $[n]$, then we have $u= v\cup w$. In other words, vertices  $v$ and $w$ have exactly one common neighbor, namely,   the vertex $u= v\cup w, $ in the   graph $ \Gamma=B(n,k)$. \\
We assert that $|g(v) \cap g(w)|= k-1$.  In fact, if $|g(v) \cap g(w)|= k-h < k-1 $, then $ h>1$ and hence $ |g(v) \cup g(w)| = k+h \geq k+ 2 $. Hence, there is no $(k+1)$-subset  $t$ in $[n]$, such that
$  g(v) \cup g(w) \subset t.  $  In other words, the vertices $g(v)$ and $g(w)$ have no common neighbors in the graph $B(n,k)$, which is impossible. Note that $f$ is an automorphism of the graph $\Gamma=B(n,k)$, hence the number of the common neighbors of $v$ and $w$ in $\Gamma$ is equal to the number of the common neighbors of $f(v)=g(v)$ and $f(w)=g(w)$ in the graph $\Gamma$.

Our argument shows that  the permutation $g=f_{|V_1}$ is an automorphism of the Johnson graph $ \Gamma_2=J(n,k) $ and therefore by [5 chapter 9, 28] there is a permutation $ \theta \in$   $Sym([n])$  such that $g= f_{\theta}$. \

On the other hand, we know that $ f_{\theta} $  by its natural action  on the vertex set of the  graph $ \Gamma= B(n,k) $  is an automorphism of $ \Gamma$.   Therefore, $l=f_{\theta }^{-1}f  $ is an automorphism of the graph $\Gamma=B(n,k)$
such that $l$ is the identity automorphism on the subset $P_1$. We now can   conclude, by Lemma 3.6.  that $l= f_{\theta }^{-1}f$,  is the identity automorphism of $ \Gamma $, and therefore $f=f_{\theta }$. \

In other words, we have proved that if $f$ is an automorphism of $\Gamma = B(n,k), $ then $f=f_{\theta}$, for some $ \theta \in$ S$ym([n] )$, and hence $f \in H$ ($n \neq 2k+1$).   We now deduce that  $G=H.$\\
(b)\ \ In this step, we  prove the theorem for the case $n=2k+1. $ Firstly,  note that   the mapping $\alpha : V(\Gamma)\rightarrow V(\Gamma) $,  defined  by the rule, $\alpha(v) = v^c$,  where
$v^c$ is  the complement of the subset $v$ in  $[n]= [2k+1]=\{ 1,2, \dots,2k+1\}$, is an automorphism of the   graph $\Gamma=B(n,k)$. In fact, if $A,B$ are subsets of $[n]$ such that   $ A\subset B$, then $ B^c\subset A^c$, and hence if \{A,B\} is an edge of the graph $ B(n,k) $, then $\{\alpha(A), \alpha(B)\}$ is an edge of the graph $ B(n,k) $.  Therefore we have,
$ < \alpha> \leq Aut(B(n,k)) $. Note that the order of $\alpha$ in the group $G=Aut(B(n,k)) $ is 2, and hence $ <\alpha> \cong \mathbb{Z}_2$.
 Let  $ H =\{ f_\theta \  |\  \theta \in$ S$ym ([n]) \} $.  We have seen already that $H  \cong$      $Sym([n]) $ and $H \leq Aut(B(n,k)) $.
 We can see that $ \alpha \not\in H $,   and for every $\theta \in$    $Sym([n])$, we have,   $f_{\theta}\alpha= \alpha f_{\theta}$. Therefore,   \\

\centerline{  $Sym([n]) \times \mathbb{Z}_2  \cong H \times <\alpha>\cong <H, \alpha>$ }
\

\centerline{ =$\{  f_{\gamma} \alpha^i\ |   \   \gamma \in$   $Sym([n]), 0\leq i \leq 1 \}=S, $}

\   \newline is a subgroup of $G=Aut(\Gamma)$. We    show that $G=S$. \

     Let $f \in Aut(\Gamma)=G$. We show that $ f \in S$. There are two cases. \

\ \ (i)  There is a vertex $v \in V_1$   such that $f(v) \in V_1$, and hence by Lemma 3.8. we have    $f(V_1)=V_1$.   \

 (ii)  There is a vertex $v \in V_1$   such that $f(v) \in V_2$, and hence by Lemma 3.8. we have    $f(V_1)=V_2$. \\
Let  $f(V_1)=V_1$. Then, by a similar argument which we did in (a), we can conclude that $f=f_{\theta}, $ where $\theta \in$   $Sym([n]).$ In other words, $f \in S.$ \\
 We now assume that $ f(V_1) \neq V_1 $. Then,  $f(V_1) = V_2$. Since the mapping $ \alpha $  is an automorphism of the graph $ \Gamma $, then $f \alpha$ is an automorphism of $\Gamma $ such that $f \alpha(V_1) = f(\alpha(V_1))= f(V_2)=V_1$. Therefore, by what is proved in (i),   we have $ f \alpha=f_{\theta} $, for some $ \theta \in$   $Sym([n]) $. Now since $ \alpha    $  is of order $2$, then $f= f_{\theta} \alpha \in S=\{  f_{\gamma} \alpha^i \ | \  \gamma \in$   $Sym([n]), 0\leq i \leq 1 \}$.

\end{proof}

Let $ \Gamma $ be a graph.  The line graph $ L( \Gamma )$
of the graph $ \Gamma$ is constructed by taking the edges of $ \Gamma$
as vertices of  $ L( \Gamma )$, and joining two vertices in $ L( \Gamma )$
whenever the corresponding edges in $ \Gamma $ have a common vertex.   It is an easy task to show that if $ \theta \in  Aut( \Gamma)$, then the mapping  $ f(\theta): V(L( \Gamma)) \rightarrow V(L( \Gamma))$ defined by the rule,
 $$f(\theta)(\{ u,v\})= \{  \theta(u), \theta(v)\}, \ \{u,v \} \in E( \Gamma ),$$
 is an automorphism of the graph  $ L( \Gamma )$. Hence, it is clear that if a graph $\Gamma$ is
 edge-transitive, then its line graph is vertex transitive. There is
an important relation between $ Aut( \Gamma)$ and  $ Aut( L(\Gamma))$. In fact, we have
the following result [2, chapter 15].
\begin{thm}
 Let $\Gamma$ be a connected graph. The mapping  $ \theta: Aut( \Gamma) \rightarrow Aut( L(\Gamma))$ defined by the rule,
$$ \theta(g) \{ u,v\} =  \{  g(u), g(v)\}, \  g \in Aut( \Gamma),  \  \{u,v \} \in E( \Gamma ), $$
is a group homomorphism and in fact we have; \\
(i) $ \theta $ is a monomorphism provided  $ \Gamma \neq K_2 $;  \\
(ii) $ \theta $ is an
epimorphism provided $\Gamma $ is not $K_4 $, $K_4 $ with one edge deleted, or $K_4 $ with two adjacent
edges deleted.
\end{thm}
Let $ [n] = \{1, ..., n \}, \   n \geq 4 $ and $ \Gamma =B(n,k) $. We let  the graph $ \Gamma_1 = L ( \Gamma ) $,  the line graph of the graph $\Gamma$.      Then,  each vertex in $ \Gamma_1 $ is of  the form $ \{ A, Ay \} $, where $ A \subseteq [n], \  |A | = k $,  $ y \in [n] - A $ and
$ Ay = A \cup \{ y \} $. Two vertices $ \{A, Ay \} $ and $ \{B, Bz \} $ are adjacent in $ \Gamma_1 $ if and only if $ A = B $ and $ y \neq z $, or $ Ay = Bz $ and $ A, B $ are distinct $ k $-subset of $ Ay $.  In other words,  if $ A = x_1 ... x_k \subseteq [n]  $,  then for  the vertex $ v = \{A, Ay_1 \} $ of $  \Gamma_1 = L ( B(n,k)),  $ we have;
$$ N ( v ) = \{ \{A, Ay_2 \}, ..., \{A, Ay_{n-k} \}, \{ Ay_1 - \{x_1\}, Ay_1 \}, ..., \{Ay_1 - \{  x_k \}, Ay_1 \} \}. $$
Hence, the graph $  L ( B(n,k)) $  is a regular graph of valency $ n-k-1+k= n-1$. In other words, the degree of each vertex in the graph    $  L ( B(n,k)) $ is independent of $k$. In fact, the graph $ L ( B(n,k)) $ is  a vertex-transitive graph, because by Proposition 3.5. the graph $  B(n,k) $\ is an edge-transitive graph.
Figure 3.  shows the graph $L(B(4,1))$ in the   plane.\\\\
\definecolor{qqqqff}{rgb}{0.,0.,1.}
\begin{tikzpicture}[line cap=round,line join=round,>=triangle 45,x=1.3cm,y=.9cm]
\clip(-1.4,-1.) rectangle (7.12,5.5);
\draw (-0.02,0.08)-- (3.08,2.);
\draw (6.56,-0.02)-- (3.08,2.);
\draw (-0.02,0.08)-- (6.56,-0.02);
\draw (-0.12,5.2)-- (2.98,3.3);
\draw (2.98,3.3)-- (6.34,5.18);
\draw (6.34,5.18)-- (-0.12,5.2);
\draw (0.98,3.08)-- (2.12,2.6);
\draw (2.12,2.6)-- (1.,1.8);
\draw (0.98,3.08)-- (1.,1.8);
\draw (4.28,2.62)-- (5.52,3.12);
\draw (4.28,2.62)-- (5.52,2.02);
\draw (5.52,2.02)-- (5.52,3.12);
\draw (-0.6,4.9) node[anchor=north west] {1,12};
\draw (6.48,4.92) node[anchor=north west] {1,13};
\draw (2.72,4.08) node[anchor=north west] {1,14};
\draw (1.04,3.7) node[anchor=north west] {2,12};
\draw (0.02,2.) node[anchor=north west] {2,23};
\draw (2.12,2.4) node[anchor=north west] {2,24};
\draw (3.76,2.34) node[anchor=north west] {4,24};
\draw (5.72,3.26) node[anchor=north west] {4,41};
\draw (5.78,2.02) node[anchor=north west] {4,43};
\draw (6.56,-0.16) node[anchor=north west] {3,34};
\draw (-0.06,-0.22) node[anchor=north west] {3,23};
\draw (2.9,1.62) node[anchor=north west] {3,31};
\draw (2.12,2.6)-- (4.28,2.62);
\draw (5.52,3.12)-- (2.98,3.3);
\draw (-0.12,5.2)-- (0.98,3.08);
\draw (5.52,2.02)-- (6.56,-0.02);
\draw (1.,1.8)-- (-0.02,0.08);
\draw (6.34,5.18)-- (3.08,2.);
\draw (1.98,-0.5) node[anchor=north west] {Figure 3.  \ \  L(B(4,1))};
\begin{scriptsize}
\draw [fill=qqqqff] (2.98,3.3) circle (1.5pt);
\draw [fill=qqqqff] (-0.12,5.2) circle (1.5pt);
\draw [fill=qqqqff] (6.34,5.18) circle (1.5pt);
\draw [fill=qqqqff] (3.08,2.) circle (1.5pt);
\draw [fill=qqqqff] (-0.02,0.08) circle (1.5pt);
\draw [fill=qqqqff] (6.56,-0.02) circle (1.5pt);
\draw [fill=qqqqff] (2.12,2.6) circle (1.5pt);
\draw [fill=qqqqff] (0.98,3.08) circle (1.5pt);
\draw [fill=qqqqff] (1.,1.8) circle (1.5pt);
\draw [fill=qqqqff] (4.28,2.62) circle (1.5pt);
\draw [fill=qqqqff] (5.52,3.12) circle (1.5pt);
\draw [fill=qqqqff] (5.52,2.02) circle (1.5pt);
\end{scriptsize}
\end{tikzpicture} \\
 Note that in the above figure $i,ij = \{  \{i \},  \{i,j  \}\}$. \

We know by Theorem 3.3.  the graph $B(n,k)$ is a  connected graph, hence its line graph, namely,  the graph $ L ( B(n,k))$  is a connected graph [4].   The graph $ L ( B(n,k))  $ has some interesting properties, for example, if $ n,k $ are odd integers, then  $ L ( B(n,k))  $ is a hamiltonian graph. In fact,    if $v$ is a vertex in $B(n,k)$, then
$deg(v) \in \{ k+1,n-k\}$,   therefore,  if $ n,k $ are odd integers, then the degree of each vertex in the graph $B(n,k)$ is an even integer, and hence $B(n,k)$ is   eulerian  [4].   Consequently, in this case  the graph   $ L ( B(n,k))$ is a hamiltonian graph.\

We know that the graph $ L ( B(n,k))$ is a vertex-transitive graph. There is a well known conjecture in graph theory which  asserts that  almost all vertex-transitive connected graphs are hamiltonian [21]. Depending on this conjecture and what is mentioned in above,
it  seems that the following conjecture has an affirmative answer.\
\

\

{\bf Conjecture} The line  graph of the graph  $B(n,k)$, namely,  $ L ( B(n,k))$ is a hamiltonian graph. \

\

A graph $\Gamma$ is called an integral graph, if all of its eigenvalues are  integers.  The notion of integral graphs was first introduced by F. Harary and A.J. Schwenk
in 1974  [12].  In general, the problem of characterizing
integral graphs seems to be very difficult. There are good surveys in this area (for example [6]).
In the scope of the  present paper,  we have the following result. \\

{\bf Fact} ( Mirafzal [26, Theorem 3.4] ) Let $n>3$ be an integer. Then,  the graph $L(B(n,1))$ is a vertex-transitive  integral graph with distinct eigenvalues,
 $ -2, -1, 0, n-2, n-1$. \\

On the automorphism group of the line graph of the graph $B(n,k), $   namely, L(B(n,k)),   by Proposition 3.3. and   Theorem 3.10. and Theorem 3.11.  we have    the following result.

\begin{thm}

Let $ n\geq 4, \
 [n] = \{1, ..., n \} $, $ 1 \leq  k < \frac {n} {2} $.  If $ \Gamma = B(n,k)$ and $ n\neq 2k+1$,
then  $ Aut (L( \Gamma )) \cong$ S$ym ([n]) $. If $ n=2k+1$, then $ Aut (L( \Gamma )) \cong$ S$ym ([n]) \times \mathbb{Z}_2$.
\end{thm}

We now  ready  to determine the values of $n,k$ such that  the graph $L(B(n,k))$ is a non-Cayley graph.

A permutation group $G$, acting on a set $V, $  is $k$-$homogeneous$  if its induced action on
$V^{ \{k\}}$ is transitive, where $V^{ \{k\}}$ is the set of all $k$-subsets of $V$. Also  we say that $G$ is $k$-$transitive$  if $G$ is transitive on $V^{ (k)}$, where $V^{ (k)}$ is  the set of $k$-tuples of distinct elements  of $V$. Note that if $G$ is $k$-homogeneous,  then we have $ {n} \choose {k}$ $ | |G|$,  and if $G$ is $k$-transitive,  then we have $ \frac{n!}{(n-k)!} |  |G|$. If the group $G$ acts regularly on $ V^{(k)} $, then $G$
is said to be $sharply$  $k$-transitive on $V$. This means that for given two $k$-tuples in $V^{(k)}$, there is a unique permutation in $G$ mapping one $k$-tuple to the other.\\
 We need  the following two results which can be found in  [8].

 \begin{thm} $[20]$ Let $G$ be a  $k$-homogeneous group
 on a finite set $\Omega$ of $n$ points, where $n \geq 2k$. Then,  if $k\geq2,$ \\
(a)  the  permutation group $G$ is $(k-1)$-transitive, and \\
(b) if also $k \geq 5,$  the given permutation group $G$ is $k$-transitive.

\end{thm}

The following result is a very deep result in group theory
\begin{thm}$[17]$
Let $G$ be a group which is  $k$-homogeneous but not $k$-transitive
 on a finite
set $\Omega$ of $n$ points, where $n \geq 2k$. Then, up to permutation
 isomorphism, one of the
following holds:\\
(i) $k$ =$ 2$ and $G \leq A\Gamma L(1,q)$ with $n = q \equiv 3$ $(mod \ 4);$\\
(ii) $k=3$ and $PSL(2,q) \leq G \leq P \Gamma L(2,q)$, where $n- 1 =q \equiv 3$ $(mod \ 4);$\\
(iii) $k$=$3$ and $G=AGL(1,8),  A \Gamma L(1,8)$ or $A \Gamma L(1,32)$; or\\
(iv) $k = 4$ and $G = PSL(2, 8), P \Gamma L(2,8)$ or $P \Gamma L(2,32)$.
\end{thm}

Note that here $q$ is a power of a prime integer, and  $A\Gamma L(1,q)$ is the group of mappings $x \mapsto a x^{\sigma}+ b$ on $GF(q)$, where $ a \neq 0$
and $b$ are in $GF(q)$ and $ \sigma \in Aut (GF(q))$. $AGL(1, q)$ consists of those mappings
with  $ \sigma= 1.$  All the groups listed in the theorem are assumed to act in their
usual permutation representations.
In the sequel,  we also need   the following result [32, chapter 7].
\begin{thm}
Let $G$ be a sharply $2$-transitive permutation group on a finite set $V. $  Then the degree of $G$ is $p^m$ for some prime $p$. Moreover, $G$ is similar to a subgroup of $Aff(W)$ which contains the translation group where $W$ is a vector space of dimension $m$ over $GF(p)$. 
\end{thm}

We are now ready to prove one the most important results of our work.

\begin{thm}
Let $n,k$ be integers, $ 3< n $, $  1 \leq k < \frac {n}{2} $ and $n\neq 2k+1$. Then the graph $L(B(n,k))$ is a non-Cayley graph if each of the following holds:\\
(i) \ \ $k \geq 4$;\\
(ii) \ \  $  k=2 $, and $n- 1 \neq q \equiv 3$ $(mod \ 4), $ where  $q$ is a power of a prime integer;\\
(iii) \ \ $k = 3$ and $n\notin \{9,33 \}$;\\
(iv) \ \ $ k=1$ and $n$ is not a power of a prime integer.

\end{thm}

\begin{proof}

On the contrary, assume that the graph $\Gamma=L(B(n,k))$ is a Cayley graph, then the automorphism group $ Aut(\Gamma) $  contains a subgroup $G$ such that $G$ acts regularly on the vertex set of $ \Gamma $ [2, chap 16]. We know by Theorem 3.12.   that $Aut(\Gamma)$ = $\{ f_{\theta} \ | \ \theta \in$ S$ym([n])  \} $. We let $G_1$=$\{   \theta \ | \ f_{\theta} \in G\}$, then $G_1$ is a subgroup of S$ym([n])$ which is isomorphic with the group $G$, and $G_1$ is $(k+1)$-homogeneous on the set $[n]=\{ 1,2,...,n   \}$. Note that  each vertex $v=\{ A,Ax \}$ in the graph  $ L(B(n,k))$ consist of a $k$-subset and a $(k+1)$-subset of $[n]$ and the group $G$ acts transitively on the vertex-set of the graph $\Gamma$. In fact,  if $A,B$ are given  $(k+1)$-subsets of the set $[n]$, then we choose   $k$-subsets  $C,D$ of $A,B$ respectively, and thus $\{A,C   \}, \{B,D   \}$ are vertices of the graph $\Gamma=L(B(n,k))$, and hence   there is some element $f_{\theta} \in G $ such that
$f_{\theta}(\{ A,C \}  )=\{ \theta(A), \theta(C)\}= \{ B,D \} $,
which implies that $ \theta(A)=B$, where $\theta \in G_1$. \

We assert that that the group $G_1$ is a $(k+1)$-transitive permutation group on the set $[n]$. In the first step, we assume that $k \geq 2. $ If $k \geq 4$, then $k+1 \geq 5$, now since $G_1$ is a $(k+1)$-homogeneous permutation group on the set $[n]=\{ 1,2,...,n   \}$, then by Theorem 3.13.  $G_1$ is a $(k+1)$-transitive permutation group on the set
 $[n]$.
 \newline If $k=3, $ then $G_1$ is a 4-homogeneous permutation group on the set $[n]=\{ 1,2,...,n   \}$. Also we have, \newline $|G_1|=|G| =| V(L(B(n,3)))|$=$ (n-3)$ $n \choose 3$=$\frac {n(n-1)(n-2)(n-3)}{6}.$
Now since $3=k < \frac {n}{2} $ and $n \neq 2k+1, $  then $n \geq 8$ and hence $|G_1| \geq$ $ \frac {(8) ( 7 )(6)(5)}{6}=280. $ On the other hand we know that  $|PSL(2, 8)|$=$9 \times 8 \times 7$=$504$,   $|P \Gamma L(2, 8)|=3(8)(8^2-1) = 3\times 7\times 8 \times 9$ and $ | P \Gamma L(2, 32)|=5(32)(32^2-1) = 5\times 31\times 32 \times 33.$   Note that if $n=9,$ then  $|V(L(B(n,3)))| =504, $ and if $n=33, $ then $|V(L(B(n,3)))| = 5\times 31\times 32 \times 33.$ 
 \newline If $k=2, $ then $G_1$ is a 3-homogeneous
 permutation group on the set $[n]=\{ 1,2,...,n   \}$. Also we have, \newline  $|G_1|=|G| =| V(L(B(n,2)))|$=$ (n-2)$$n \choose 2$=$\frac {n(n-1)(n-2)}{2}$.  On the other hand we have  $|AGL(1, 8)|=56,$
 $|A \Gamma L(1, 8)|=168, $ and $|A \Gamma L(1, 32)|=32 \times 31 \times 5. $ Note that if $n=8, $ then $|V(L(B(8,2)))|$ = 168, but $8- 1 =7 \equiv 3$ (mod \ 4).

Now, by comparing the order of the group $G_1$ with orders of  the groups which appear in the cases (ii),(iii),(iv) in Theorem 3.14. we conclude by Theorem 3.14. that the group $G_1$ is a $(k+1)$-transitive permutation group on the set $[n]$.\

 Therefore, if
$u_1=( x_1,x_2,...,x_k,x_{k+1})$  is a $(k+1)$-tuple  of distinct elements of $[n]$, then for  $ u_2=( x_1,x_{k+1},x_k,...,x_3,x_2)  $   there is an element $\theta \in G_1$ such that,  \
 $$\theta (u_1) = (\theta(x_1), \theta(x_2),...,\theta(x_k), \theta (x_{k+1} ) )= u_2=( x_1,x_{k+1},x_k,...,x_3,x_2) $$
\
 Now, for the $k$-subset $ A= \{ x_2,...,x_k,x_{k+1}  \} $ we have;
$$\theta(A)= \{ \theta(x_2), \theta(x_3),...,\theta(x_k), \theta (x_{k+1} )     \} =A$$
  Note that if $k=1$, then $\theta$ can be the identity element of the group $G_1$, but if $ k>1$, then $\theta \neq 1$, and hence in the first step we assume that $k>1$. \newline
Therefore,  if we consider the $k$-subset $A=\{    x_2,...,x_k, x_{k+1}\}$ of $[n]$, then $v= \{A,Ax_1\}$ is a vertex of the graph $\Gamma=L(B(n,k))$, and thus, for the element $ f_{\theta } $ in the group $G$  we have
$$ f_{\theta }(v) =\{ \theta(A), \theta(Ax_1) \} = \{A,A\theta(x_1)   \} = \{A,Ax_1\}=v, $$
which is a contradiction, because $1 \neq f_{\theta } \in G_v$ and the group $G$ acts regularly on the vertex set of the graph $\Gamma$. Consequently, if $ k>1$, then the graph $ L(B(n,k))  $ is a vertex-transitive non-Cayley graph. \

We now assume that $k=1$.  Before proceeding  we mention that, in the sequel,  we do not luckily  need  Theorem 3.14.    because,  this theorem does not work in this case. We know that $ | V  | = n(n-1)$ and $G$ is a subgroup of $Aut (\Gamma)$ which is regular on the set $V$, thus $ |G |= n(n-1)$ and hence $| G_1| =n(n-1)$. We assert that $G_1$  is 2-transitive on the set $[n]$. If $(i,j)$ and $(r,s)$ are 2-tuples of distinct elements of $[n]$, then $\{i, ij \}$ and $\{r, rs \}$ are vertices of the graph $L(B(n,1))$,  and thus there is some $ f_{\phi} $ in $G$ such that
$$  f_{ \phi }(\{i, ij \}) =\{\phi (i),\phi (i)\phi (j)\} =  \{r, rs\}, $$
which implies that for $ \phi \in G_1 $ we have $ \phi (i)=r, \phi (j)=s $, namely $\phi (i,j)=(r,s)$. \

Therefore, $G_1$ is a 2-transitive group on $[n]$ of order $n(n-1)$, and hence $G_1$ is sharply 2-transitive on $[n]$.
Therefore, by Theorem 3.15. the integer $ n$ is of the form  $p^m$ for some prime $p$. In other words, if $ n $ is not a power of a prime, then the graph $L(B(n,1))$ has no subgroup $G$ in its automorphism group such that $G$ acts regularly on the vertex set of $L(B(n,1))$, and hence this graph is not a Cayley graph.
\end{proof}

\begin{rem} Note that if $n=2k+1$, then by Theorem 3.10 we have $Aut(B(n,k)) \cong$       $Sym([n]) \times \mathbb{Z}_2$, hence  Theorem 3.16 does not work in this case. We can   solve the problem of cayleyness of the graph $L(B(2k+1,k))$ in some cases, but the method for dealing with this problem is completely different from what is appeared in Theorem 3.16.  and hence   we will discuss on this case in the last part of  our paper.  

\end{rem}

 Note that Theorem 3.16  does not say anything if $k=1$ and $n$ is a power of a prime, but we can show that the graph $L(B(4,1))$ which is displayed in   Figure 4. is a Cayley graph. In fact we have the following result.

\begin{prop}
Let $G= A_4$, the alternating group of degree $4$ on the set $[4]$=$  \{  1,2,3,4  \} $. Let $\rho =(1,2,3)$ and $ a=(1,2)(3,4)$. If $\Gamma$ is the Cayley graph $Cay(G;S)$ where $S=\{ \rho, { \rho}^2, a \}$, then $\Gamma$ is isomorphic with $B(L(4,1))$. In other words,   $L(B(4,1))$  is a Cayley graph.
\end{prop}

\begin{proof}
Let $b=(1,3)(2,4), \   c=(1,4)(2,3)$. Now, since $H= \{ 1, \rho, {\rho}^2 \}$ is a subgroup of $G$, then $G=H \cup aH \cup bH \cup cH$. Note that $ K= \{ 1,a,b,c \}$, is an abelian subgroup of the group $G$, with the property that the order of each non-identity element of $G$ is 2, and in $G$ we have $ ab=c,\  bc=a, \  ca=b$. \

We know  that in the Cayley graph $Cay(G;S)$, each vertex $v$ is adjacent to every vertex $vs$, $s \in S$ and hence if $v\in \{ a,b,c \}$, then $ N(v)= \{va, v \rho, v{\rho}^2  \} $, where $N(v)$ is the set of neighbors of $v$. Now a simple computation shows that  Figure 3,   displays the graph
$Cay(A_4;S) $ in the plane,  is isomorphic with the graph
$L(B(4,1))$.  In fact, we have, \\
$ { \rho }^{-2}c{ \rho }^{2} $=$ {\rho }^{-2}(1,4)(2,3){ \rho }^{2} $=$ (3,4)(1,2)$=$a \in S$, which implies that $ { \rho }^{2}$ is adjacent to $c { \rho }^{2}.$  \\
$ {(a \rho )}^{-1}c{ \rho } $=$ {\rho }^{-1}(ac){ \rho } $= $ {\rho }^{-1}b{ \rho } $=${\rho }^{-1} (1,3)(2,4){ \rho } $=$a \in S$, which implies that $ {a \rho }$ is adjacent to $c { \rho }$.\\
$( a {\rho }^{2})^{-1}b{ \rho }^{2} $=
$ { \rho }^{-2}ab{ \rho }^{2} $=
$ { \rho }^{-2}c{ \rho }^{2} $=
$ { \rho }^{-2}(1,4)(2,3){ \rho }^{2} $=
$(3,4)(1,2)$=
$ a \in S$, which implies that $ {a \rho }^{2}$ is adjacent to $b { \rho }^{2}$. \\
$cb=a \in S$ which implies that $c$ is adjacent to $b$. \\
Also, note that $H,a H,bH,cH$, are 3-cliques in  $Cay(G;S)$.

\end{proof}

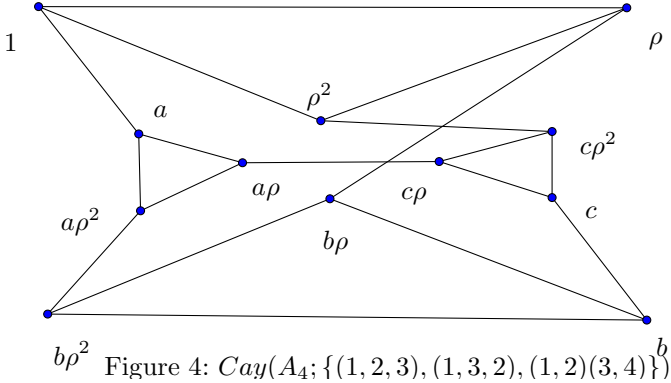
\definecolor{qqqqff}{rgb}{0.,0.,1.}
\begin{tikzpicture}[line cap=round,line join=round,>=triangle 45,x=1.2cm,y=.8cm]
\clip(-0.7,-1.) rectangle (7.12,5.5);
\draw (-0.02,0.08)-- (3.08,2.);
\draw (6.56,-0.02)-- (3.08,2.);
\draw (-0.02,0.08)-- (6.56,-0.02);
\draw (-0.12,5.2)-- (2.98,3.3);
\draw (2.98,3.3)-- (6.34,5.18);
\draw (6.34,5.18)-- (-0.12,5.2);
\draw (0.98,3.08)-- (2.12,2.6);
\draw (2.12,2.6)-- (1.,1.8);
\draw (0.98,3.08)-- (1.,1.8);
\draw (4.28,2.62)-- (5.52,3.12);
\draw (4.28,2.62)-- (5.52,2.02);
\draw (5.52,2.02)-- (5.52,3.12);
\draw (-0.6,4.9) node[anchor=north west] {1};
\draw (6.48,4.92) node[anchor=north west] {$\rho$};
\draw (2.72,4.08) node[anchor=north west] {${\rho}^2$};
\draw (1.04,3.7) node[anchor=north west] {$a$};
\draw (0.02,2.) node[anchor=north west] {$ a {\rho}^2$};
\draw (2.12,2.4) node[anchor=north west] {$  a \rho$};
\draw (3.76,2.34) node[anchor=north west] {$ c\rho$};
\draw (5.72,3.26) node[anchor=north west] {$c{ \rho}^2$};
\draw (5.78,2.02) node[anchor=north west] {$c$};
\draw (6.56,-0.16) node[anchor=north west] {$b$};
\draw (-0.06,-0.22) node[anchor=north west] {${ b \rho}^2$};
\draw (2.9,1.62) node[anchor=north west] {$ b\rho$};
\draw (2.12,2.6)-- (4.28,2.62);
\draw (5.52,3.12)-- (2.98,3.3);
\draw (-0.12,5.2)-- (0.98,3.08);
\draw (5.52,2.02)-- (6.56,-0.02);
\draw (1.,1.8)-- (-0.02,0.08);
\draw (6.34,5.18)-- (3.08,2.);
\draw (0.5,-0.45) node[anchor=north west]
 {Figure\  4:  $Cay (  A_4;\{ (1,2,3), (1,3,2), (1,2)(3,4) \}  ) $};
\begin{scriptsize}
\draw [fill=qqqqff] (2.98,3.3) circle (1.5pt);
\draw [fill=qqqqff] (-0.12,5.2) circle (1.5pt);
\draw [fill=qqqqff] (6.34,5.18) circle (1.5pt);
\draw [fill=qqqqff] (3.08,2.) circle (1.5pt);
\draw [fill=qqqqff] (-0.02,0.08) circle (1.5pt);
\draw [fill=qqqqff] (6.56,-0.02) circle (1.5pt);
\draw [fill=qqqqff] (2.12,2.6) circle (1.5pt);
\draw [fill=qqqqff] (0.98,3.08) circle (1.5pt);
\draw [fill=qqqqff] (1.,1.8) circle (1.5pt);
\draw [fill=qqqqff] (4.28,2.62) circle (1.5pt);
\draw [fill=qqqqff] (5.52,3.12) circle (1.5pt);
\draw [fill=qqqqff] (5.52,2.02) circle (1.5pt);
\end{scriptsize}
\end{tikzpicture} \

\

\

We now want to show that if $ n= p^m $,  for some prime $p$,
then the line graph of $ B(n,1) $, namely,  the graph $ L(B(n,1)) $ is a Cayley graph. \newline
In the first step, note that in $L(B(n,1))$,  the subgraph induced
by the set $C_i=  \{ \{i,ij \}  \ | \ j \in [n], j\neq i \}  $, $ 1\leq i \leq n   $,  is a $(n-1)$-clique. Also,  $  V(L(B(n,1))=(\cup C_i)_{i\in [n]} $
and for each pair of cliques $C_i, C_j$ in $ L(B(n,1) $ there is exactly one pair  of vertices $  v_i \in C_i, v_j \in C_j $, namely,  $\{i,ij\}, \{j,ji\}$, such that $v_i$ is adjacent to $v_j$. In fact,  every regular graph $\Gamma$ of order $n(n-1)$  and valency $n-1$ with these properties is isomorphic with  $ L(B(n,1)) $. 

\begin{thm}
Let $\Gamma$ be a regular graph of  order  $ n(n-1)$ and valency $n-1$. Suppose that in $\Gamma$ there are $n$ disjoint $(n-1)$-cliques $D_1,..., D_n$, such that $  V(\Gamma)=(\cup D_i)_{i\in [n]} $ and for each pair of distinct  cliques $D_i, D_j$ in $ \Gamma $ there is exactly one pair of vertices $ v_i,v_j$ such that $ v_i \in D_i, v_j \in D_j $  and  $v_i$ is adjacent to $v_j$. Then, $\Gamma$ is isomorphic with $L(B(n,1))$.

\end{thm}

\begin{proof}
In the first step, note that each vertex $v$ in a   $(n-1)-$clique $D$ is adjacent to exactly one vertex $w$ which is not in $D$, because $v$ is of degree $n-1$. We now,  choose one of the cliques in $\Gamma$ and label it by $C_1$. we let vertices in $C_1$ are  $\{v[1,12],...,v[1,1n]\}$. If $D$ is a clique in $\Gamma$ different from $C_1$, then there is exactly one integer $j, \  2\leq j \leq n$, such that $v[1,1j]$ is adjacent to exactly one vertex of $D$, say, $v_D$. Then, we label the clique $D$ by $C_j$. We now label the vertices in $C_j$ as follows,  \newline
 If $v \in C_j$, then there is exactly one $i, i\in \{1,2,...,n  \}, i \neq j$ such that $v$ is adjacent to exactly one vertices in $C_i$, now in such a case,  we label $v$ by $v[j, ij]$. \

Now, it is an easy task to show that the mapping $\phi: V(L(B(n,1))   \longrightarrow
V(\Gamma)$, defined by the rule, $\phi(\{j,ij\}) = v[j,ij]$ is a graph isomorphism.
\end{proof}
\
Let $ H,  K$ be  groups, with $H$ acting on $ K$ in such   way that the group structure of
$K$ is preserved (for example $H$ is a subgroup of automorphisms of the group $K$). So for each $u \in K$ and  $x \in  H$ the mapping $ u \longrightarrow u^x$
is an automorphism
of $K$ (Note that the action of $H$ on $K$ is not specified directly).
The semi-direct product of $K$ by  $H$ denoted by
$K \rtimes H $ is the set    \\$$ K \rtimes H   = \{ (u, x ) \ | \  u\in  K,  x \in H   \}, $$ \\
with binary operation $(u, x )(v , y ) = (uv^{x^{-1}}, xy)   $

\begin{thm}
if $ n= p^m $,  for some prime $p$,
then the line graph of $ B(n,1) $, namely,  the graph $ L(B(n,1)) $ is a Cayley graph. \

\end{thm}

\begin{proof}
Consider the finite field $GF(n) $ and let $K$ be the group $K=(GF(n),+)$.  For each $ 0 \neq a \in K $, we define the mapping $f_a :K
\longrightarrow K $, by the rule $ f_a(x)=ax, \  x\in K$. Then, $f_a$ is an automorphism of the group $K$ and
 $ H= \{ f_a \  | \ 0 \neq a\in K \}$ is a group (with composition of functions) of order $n-1$ which is isomorphic with  the multiplicative
group of the field $GF(p^m)$.  If we let $G= K \rtimes H$, then $G$ is a well defined group of order $n(n-1)$. Note that $G$ is not an abelian group. Let $T= \{ (0,h) \ | \ h\in H \}$. Then,  $T$ is a subgroup of order $n-1$ in the group $G$ which is isomorphic with $H$ and $[G:T]=n$,  where $[G:T]$ is the index of $T$ in
$G$. Hence, there are elements $b_1,...,b_n$ in $G$ such that $G=b_1T \cup...\cup b_nT, $  and if $i \neq j $, then $ b_iT \cap b_jT= \emptyset $.  Note that $f_{-1} \in H$ and $   { f_{-1}^2=i },   $   where $i$ is the identity element of $H$.
Then, for the element    $ \alpha=(1,f_{-1}) $, we have  \\
$${\alpha}^2= (1,f_{-1})(1,f_{-1})  =  (1+(-1)1, (f_{-1})^2) =    (0,i) =e, $$\\
where $e$ is the identity element of $G$.   Note that $\alpha \notin T$. \

We now let $ \Gamma= Cay(G;S),  $  where $S=(T-\{  e\}) \cup \{ \alpha \}$.  Note that $ S= S^{-1},  $  because $T$ is a subgroup of $G$ and ${\alpha}^2=1$ (and consequently ${\alpha}^{-1}=\alpha \in S$).   Since, in the graph $ \Gamma $ vertices $x,y$ are adjacent if and only if $x^{ -1 }y \in S$,  then for each $i,    1\leq i \leq n, $ the subgraph induced by the set $C_i = b_iT$, is a $(n-1)$-clique in $\Gamma$. Since $\alpha \notin T$, then
$C_i \cap C_i \alpha = \emptyset$. In fact, if $v=b_it= b_it_1 \alpha, \  t_1,  t \in T$, then  $ t= t_1 \alpha$, and hence
$ \alpha \in T$ which is a contradiction. Therefore, if $v\in C_j$, then $v$ is adjacent to $v \alpha$ and   $v \alpha \in C_i,$ for some $ i,  i \neq j$. Since $\Gamma$ is a regular graph of valency
$n-1$, then for each vertex $v\in C_j$ there is exactly one $i, i \neq j$ such that $v$ is adjacent to exactly one  vertex $w$ in $C_i$. Now,  by Theorem 3.19. we conclude  that $\Gamma$
 is isomorphic with $ L(B(n,1))$, and therefore the graph $  L(B(n,1))$ is a Cayley graph. \\
 \end{proof}

We now  discuss on the cayleyness of the graph $L(B(2k+1,k)).$  As we have stated in Remark 3.17. it seems that the method of Theorem 3.18. does not work in this case. In the sequel, we need the following fact  which appeared in Mirafzal [25]. 

\begin{lem} $[25]$ Let $ k > 2 $ be an even integer  such that $k$ is not of the form $ k = 2^t  $  for some $ t \geq 2$,  then
the number  $ { 2k+1}  \choose { k } $ is a multiple of $ 4 $.
\end{lem}
In the following theorem  we prove,  by an elementary method,  that for \lq{}almost all\rq{} even integers $k$, the line graph of the graph   $B(L(n,k))$ is a vertex-transitive non Cayley graph.

\begin{thm} Let $ k > 2 $ be an even integer  such that $k$ is not of the form $ k = 2^t  $  for some $ t \geq 2$.  Then,  the line graph of the graph   $B(2k+1,k))$, that is, $L(B(2k+1,k))$ is a vertex-transitive non Cayley graph. 

\end{thm}

\begin{proof}\ We  let $n=2k+1. $   We know by Theorem 3.10. and Theorem 3.12.   that \newline
 $ Aut(L(B(n,k)))
=  S =\{  f_{\gamma} \alpha^i\ |   \   \gamma \in Sym([n]), 0\leq i \leq 1 \} ( \cong Sym([n]) \times \mathbb{Z}_2 )$, where $ \alpha $ and $f_{\gamma}$ are automorphisms of the graph $\Gamma  = L(B(n,k))$ which are defined in Theorem 3.10.
   Suppose that
 $ \Gamma$ is a Cayley graph.  Then,  $ Aut(\Gamma)$  has a subgroup $ R
 $ such that $R$
   acts regularly on  the set $ V(\Gamma)$.
 Then $ |R| = |V(\Gamma)|$  = ${n \choose k}$ $(k+1)$ =$ \frac{n!}{(k!)(n-k)!}$$(k+1)$.  Since  $k$ is not of the form $ k = 2^t  $  for some $ t \geq 2$,  then by Lemma 3.21. 
the number  $ { 2k+1}  \choose { k } $ = $ { n}  \choose { k } $ is a multiple of $ 4 $, hence $|R|$ is a multiple of 4. 
 If $ r$  is an element of $ R $, then by Theorem 3.10.  $ r$ has a form such as $ f_{ \sigma} \alpha^i $,
where $  \sigma \in Sym([n]) $ and   $ i \in \{ 0,1\}$. We know  that $\alpha f_{\sigma} = f_{\sigma} \alpha$, for every $ \sigma \in Sym([n]) . $  If $f_{ \sigma} \alpha^i \in R $, then 
  $$ (f_{ \sigma} \alpha^i) (f_{ \sigma} \alpha^i)=f_{ \sigma}f_{ \sigma}(\alpha^i)^2 ={f_{ \sigma}}^2 =f_{{ \sigma}^2}  \in R.$$
Then there are  elements of the form $ f_{ \theta}$, $\theta \in Sym([n]) $ in $R$.
     Let $ M_1 =  \{f_{ \phi} \ | \
f_{ \phi} \in R \}$.  We can easily see that $ M_1 $ is a subgroup of $ R $. In this step, there are two cases, that is,  (a) $M_1=R$, and (b) $M_1\neq R.$\

    (a) If $M_1= \{f_{ \phi} \ | \
f_{ \phi} \in R \} =   R, $ then, by using the method which was used in  Theorem 3.16. we can show that the graph $L(B(2k+1,k))$ is a vertex transitive non Cayley graph. 

(b)  We now assume that $M_1 \neq R.$  Hence,   $ R  $ contains
elements of the form $f_{ \theta} \alpha. $ 
 We let  $ M_2$ =$
\{
 f_{\theta}\alpha \  |  \ f_{\theta}\alpha  \in R \}.$ Let $f_{ \theta_0}\alpha$ be a fixed element of $M_2. $
  Then $ M_2 f_{\theta_0}\alpha  \subseteq M_1 $,
 because $ (f_{ \theta} \alpha)( f_{\theta_0}\alpha ) $ = $  f_{ \theta}f_{\theta_0} (\alpha)^2$ = $  f_{ \theta}f_{\theta_0}$ =  $  f_{ \theta\theta_0}.$
 Then, $ | M_2| \leq |M_1| $. Since $ M_1 f_{ \theta_0} \alpha   \subseteq M_2
 $, then $ |M_1| \leq |M_2|$, and hence $ |M_1| = |M_2| = (1/2) |R|$= $\frac{1}{2}$${n \choose k}$$(k+1).$ Now since $|R|$ is a multiple of 4, then $|M_1|$ is an even integer.    Since $|M_1|$ is an even integer, then $ 2 $ divides $|M_1| $.  Therefore,   by the Cauchy's
theorem the group $ M_1 $ has an element $ f_{\theta} $  of order $ 2 $ where $\theta \in Sym([n])$. Note that the orders of $\theta$ and $f_{\theta} $ are identical, so $\theta$ is of order 2 in the group $Sym([n])$.
We know that each element of $ Sym([n])$ has a unique factorization into disjoint
cycles of $ Sym([n]) $, hence
we can write $ \theta = \rho_1 \rho_2 ... \rho_{h} $, where
each $ \rho_i $ is  a cycle  of $ Sym ( [n] ) $ and
 $ \rho_i \cap \rho_j ={\emptyset} $
when $ i \neq j $. We also know that if $ \theta = \rho_1 \rho_2 ... \rho_{h} $, where
each $ \rho_i $ is  a cycle of $ Sym ( [n] ) $ and
 $ \rho_i \cap \rho_j ={\emptyset} $, then the order
 of the permutation $ \theta $ is the least common multiple
 of the integers,  $ | \rho_1 |, \  |  \rho_2 |,
   ... , | \rho_h  | $.     Since $ \theta $ is of order $ 2 $, then the order
of each $ \rho_i $ is $ 2 $ or $ 1 $, namely,  $ | \rho_i | \in \{1, 2 \} $.
In   other words,  each $ \rho_i $ is a transposition or a cycle of length $ 1 $. \
Let $ \theta = \tau_1 \tau_2 ... \tau_a ( i_1) (i_2) ... (i_b) $,  where
each $ \tau_r $ is a transposition and each $ i_s \in [n] $.  Therefore, we have $ 2a+b =n=2k + 1 $, where $b$ is an odd integer, and hence it is non-zero.
Since $ b  $ is a positive odd integer, then $ b-1$ is an
 even integer. we let $ d= \frac { b-1} {2}$, so that $ d $ is a non-negative
 integer,  $ d < b$ and $ k= a+d $. Let $ \tau_r = (x_r y_r )$, $ 1 \leq r \leq a$,  where $ x_r, y_r \in [n] $.
 Now,  there are two cases: 
  (i) \  \    $ 2a \leq k$, \  \ (ii) \ \ $ 2a > k$.

 (i) \ \ Suppose   $ 2a \leq k$. Then there is some non-negative integer $t$
 such that $ 2a+ t =k$.  
   Thus,   for transpositions
 $ \tau_1,  \tau_2,  ...,  \tau_{a}$ and cycles $ ( i_1),...,( i_t )$ of
 the cycle factorization
 of $ \theta $, the set  $ u= \{ x_1, y_1,..., x_a, y_a,  i_1,  i_2,  ...,  i_t    \} $
  is a $k $ subset of the set $ [n]$. Hence, $v =  \{  u,u\cup\{ i_b\} \}$ is a vertex of the graph $\Gamma = L(B(n,k))$.    Therefore,   we have; 
 $$ f_{\theta}( v ) = \{ \{ \theta ( x_1), \theta ( y_1 ), ..., \theta ( x_{a}),\theta ( y_{a}),\theta ( i_1),..., \theta ( i_t)\},$$ 
$$\{ \theta ( x_1), \theta ( y_1 ), ..., \theta ( x_{a}),\theta ( y_{a}),\theta ( i_1),..., \theta ( i_t), \theta ( i_b)\}\}= $$
 $$ \{ \{ y_1, x_1, ..., y_{a}, x_{a}, i_1,  i_2,  ...,  i_t \},\{ y_1, x_1, ..., y_{a}, x_{a}, i_1,  i_2,  ...,  i_t,i_b \}\}=$$ $$  \{  u,u\cup\{ i_b\} \} = v$$ 
 \  \ \ \ \ (ii) \   Suppose $ 2a > k$. We know by the assumption that $k$ is an even integer. Let $k=2m$, where $m \geq 2$ is an integer. Thus $2a > 2m$ and hence $ a > m$.   Then by transpositions $ \tau_1,  \tau_2,  ...,  \tau_{m}$ we can construct   the $k$-subset $ u= \{ x_1, y_1,..., x_m, y_m\} $ of $[n]$,  and hence $v =  \{  u,u\cup\{ i_b\} \}$ is a vertex of the graph $\Gamma = L(B(n,k))$, thus we have,    
$$ f_{\theta}( v ) =\{ \{ \theta ( x_1), \theta ( y_1 ), ..., \theta ( x_m), \theta (y_m)\},$$
 $$\{ \theta ( x_1), \theta ( y_1 ), ..., \theta ( x_m), \theta (y_m),\theta ( i_b)\} \}=
  \{  u,u\cup\{ i_b\} \} = v $$
 
From the above argument, it follows that the automorphism $ f_{\theta} $, which is an element of the subgroup $R$,  fixes a vertex of the   graph $ L(B(n,k)) $  which is a contradiction, because $R$ acts regularly on the vertex-set of $ L(B(n,k)) $ whereas  $ f_{\theta} $ is not the identity element.
\end{proof}

 \section{ Conclusion}
 In this paper, we determined the automorphism group of the graph $ {Q_n}(k,k+1) $,   where $ {Q_n}(k,k+1) $  is the subgraph of the hypercube $Q_n$  which is induced by the set of vertices of weights $k$ and $k+1$, for all $n >3$ and $  0< k < \frac{n}{2} $.  Then, we studied some algebraic properties of the line graph of these graphs. In particular,  we proved that    if $k\geq 3$  and $  n \neq 2k+1$, then  except for the cases $k=3, n=9$ and $k=3, n=33$, the line graph of the graph $ {Q_n}(k,k+1)  $ is a vertex-transitive non Cayley graph. Also, we showed that the line graph of the graph   $ {Q_n}(1,2) $ is a Cayley graph if and only if $ n$ is  a power of a prime $p$. Moreover, we showed that if $k$  is an integer such that it is not a power of 2, then the line graph of the graph   $ {Q_{2k+1}}(k,k+1) $ is a vertex-transitive non-Cayley graph.   On the other hand,  in the following cases, we do not know whether   the line graph of the graph ${Q_n}(k,k+1)$ is a Cayley graph; \\
(1) $  k=2 $, and $n- 1 = q \equiv 3$ (mod \ 4),   where  $q$ is a power of a prime integer, \\
(2) $k = 3$ and $n\in \{9,33 \}$, \\
(3) $k = 2^t$ or $k$ is an odd integer,  when $n=2k+1$.

\end{document}